\documentclass[12pt,a4paper]{amsart}
\usepackage{amsfonts,enumerate}
\usepackage{amssymb}
\usepackage{color}

\numberwithin{equation}{section}

 \def\numset#1{{\mathbb #1}}

\theoremstyle{plain}
\newtheorem{Th}{Theorem}[section]
\newtheorem{Lemma}[Th]{Lemma}
\newtheorem{Cor}[Th]{Corollary}

 \theoremstyle{definition}
\newtheorem{Def}[Th]{Definition}
\newtheorem{Conj}[Th]{Conjecture}
\newtheorem*{Rem}{Remark}
\newtheorem{?}[Th]{Problem}

\newcommand{\kp}[1]{\eqref{#1}} 

\newcommand\setZ{\numset Z}

\newcommand{\Z}{\mathbb{Z}}

\newtheorem*{claim}{Claim}
\newtheorem*{cdimo}{Proof of claim}

\newcommand{\bdef}{\begin{definizione}}		
\newcommand{\bteo}{\begin{Th}}			
\newcommand{\eteo}{\end{Th}}			
\newcommand{\bclaim}{\begin{claim}}			
\newcommand{\eclaim}{\end{claim}}			
\newcommand{\bconj}{\begin{Conj}}			
\newcommand{\econj}{\end{Conj}}			
\newcommand{\boss}{\begin{Rem}}			
\newcommand{\eoss}{\end{Rem}}			
\newcommand{\bdimo}{\begin{proof}}			
\newcommand{\edimo}{\end{proof}}			
\newcommand{\bcdimo}{\begin{cdimo}}			
\newcommand{\ecdimo}{\end{cdimo}}			
\newcommand{\blem}{\begin{Lemma}}			
\newcommand{\elem}{\end{Lemma}}			
     
\begin{document}
\bibliographystyle{amsplain}

\title{Carries and the arithmetic progression structure of sets
}

\author{Francesco Monopoli}       
\address{Dipartimento di Matematica, Universit\`a degli Studi di Milano\\
     via Saldini 50, Milano\\
     I-20133 Italy}

\email{francesco.monopoli@unimi.it}

\author{Imre Z. Ruzsa}
\address{Alfr\'ed R\'enyi Institute of Mathematics,\\
  Hungarian Academy of Sciences\\
     Budapest, Pf. 127\\
     H-1364 Hungary}

\email{ruzsa@renyi.hu}
\thanks{Author was supported by ERC--AdG Grant No. 321104 and
Hungarian National Foundation for Scientific
Research (OTKA), Grants No. 109789, 
and NK104183.}


     \maketitle

     \section{Introduction}

     If we want to represent integers in base $m$, we need a set $A$ of digits, which needs to be a complete set of residues
     modulo $m$. The most popular choices are the integers in $[0, m-1]$ and the integers in $(-m/2, m/2]$.

 When adding two integers with last digits $a_1, a_2 \in A$, we find the unique $a\in A$ such that
     \[ a_1+a_2 \equiv a \pmod m , \] which will be the last digit of the
     sum, and $(a_1 + a_2 -a)/m$ will be the \emph{carry}. Diaconis, Shao and Soundararajan in the nice paper \cite{diaconis2014carries} and Alon in \cite{alon} show that the above two popular sets
both have an extremal property: $(-m/2, m/2]$ minimizes the number of pairs $a_1, a_2$ for which there is a nonzero carry, while
$[0, m-1]$ minimizes the number of distinct carries, and both examples are unique up to certain linear transformations.

The second extremal property is essentially equivalent to the following statement:

\emph{Let $A\subset \setZ_{m^2} $ be a set which forms a complete set of residues modulo $m$. If $A+A \subset A + \{x,y\}$ with some  
$x,y\in \setZ_{m^2} $, then $A$ is an arithmetic progression.}

In \cite{diaconis2014carries} this is proved for the case $m$ prime. 

From now on we call a set $A\subset\setZ_q$ a \emph{digital set,} if $m=|A|$ satisfies $m|q$, and $A$ is a complete set of residues modulo $m$.
A more general claim could sound as follows:

\emph{Let $A\subset \setZ_{m^2} $ be a digital set with  $|A| =m$. If $|A+A| \leq 2m$,
then $A$ is an arithmetic progression.}

In \cite{Grynkiewicz} we find a complete description of finite sets in commutative groups satisfying  $|A+A| \leq 2|A|$. This could be used to deduce the
above claim. This deduction is not immediate, however, as this description contains a lot of subcases.
We remark also that in \cite{Hamidoune} Hamidoune, Serra and  Z{\'e}mor prove a result somehow similar to our Theorem \ref{main}, albeit with a restriction on $k$ and with different hypotheses, which could be used to prove the claim.

The aim of this paper is to provide a further generalization of the following form:

\emph{Let $A\subset \setZ_{m^2} $ be a digital set with  $|A| =m$. For every set $B$
such that $ 1 < |B| < m^2-m$ we have $ |A+B| > m + |B|$, with certain exactly described exceptions.}

Digital sets of cardinality $m$ exist in $\setZ_q$ whenever $m|q$.
For our arguments we need a stronger assumption, which is, however, more general than the case  $q=m^2$, namely, that
$m$ and $q$ are composed of the same primes, and the exponent of each prime in $q$ is strictly greater than in $m$.
This is a natural restriction, as otherwise there are digital sets that are either contained in a nontrivial subgroup,
or are unions of cosets of a nontrivial subgroup.

As we are looking for estimates that depend only on the cardinality of the other set  $B$, it is comfortable to express this
in terms of the \emph{impact function} of the set $A$:
 \[ \xi(n) = \xi_A(n) = \min_{|B| =n} |A+B| , \]
defined for integers $n$ that can serve as cardinality of a set; if we are in $ \setZ_q$, this means $|B| \leq q$.

Some values of $\xi$ are determined by the size of $A$: we have $\xi(0)=0$, $\xi(1)=m$, $\xi(n)=q$ for $q-m < n \leq q$ and
$\xi(q-m)=q-1$ by a familiar pigeonhole argument. A nontrivial estimate may exist for $1<n<q-m$. The case $n=2$ can be interpreted
via the arithmetic progression structure of $A$. Given any $t\in \setZ_q \setminus  \{0 \} $, $A$ can be decomposed as the union of
some cosets of the subgroup generated by $t$ and some arithmetic progressions of difference $t$. Let $\alpha_t(A)$ be the number of
arithmetic progressions in this decomposition. We have clearly
 \[ | A + \{x, x+t\}| = m + \alpha_t(A) \]
for every $x$, hence
  \[ \xi(2) = \min_t \alpha_t(A) . \]

Thus, $\xi(2)>m+2$ holds unless $A$ is the union of at most two arithmetic progressions (as we shall soon see, digital sets
 do not contain nontrivial cosets). Hence the strongest result of this kind that may hold (save the bound 15) sounds as follows.

\bteo\label{digsetteo}
  Let $q$ and $m$ be positive integers  composed of the same primes such that the exponent of each prime in $q$ 
  is strictly greater than in $m$. Let $A\subset \setZ_{q} $ be a digital set with  $|A| =m > 15$.
 We have
   \[ \xi_A(n) > m+n \]
for $1<n<q-m$, unless $A$ is the union of at most two arithmetic progressions with a common difference.
\eteo

A description of sets satisfying $|A+A| \leq 2m$ could be achieved by analyzing unions of two arithmetic progressions,
a task not difficult which allows us to generalize the aforementioned result found in \cite{diaconis2014carries}.

\begin{Cor}\label{digsetcorollary}
  Let $q$ and $m$ be positive integers  composed of the same primes such that the exponent of each prime in $q$ 
  is strictly greater than in $m$. Let $A\subset \setZ_{q} $ be a digital set with  $|A| =m > 15$ such that $2A \subseteq \{x,y\}+A$ for some $x, y \in \Z_q$. Then there exist $c \in (\Z_q)^\times$ and $d \in q\Z_q$ such that either $cA+d = \{ 0, 1, \dots, q-1\}$ or $cA+d = \{ 1, 2, \dots, q\}$.
\end{Cor}

In the first part of the paper we prove a somewhat weaker result. It turns out that the key to the conjecture above
would be to understand (i) the cases when $\xi(2)=\xi(3)$, (ii) the cases when the decomposition of our set into the 
minimal $\xi(2)$ arithmetic progressions is not unique. The second part of the paper is devoted to these questions,
including the proof of Theorem \ref{digsetteo}.

Our main result is as follows.

\begin{Th} \label{main}
  Let $q$ and $m$ be positive integers  composed of the same primes such that the exponent of each prime in $q$ 
  is strictly greater than in $m$. Let $A\subset \setZ_{q} $ be  a digital set with  $|A| =m$ and $\xi=\xi_A$
  its impact function. Let $k$ be a nonnegative integer. If the inequality
   \[ \xi(n) \geq n+m+k \]
   holds in the range
    \[ 2 \leq n \leq \frac{3+\sqrt{16k+1}}{2} \]
    and $m > m_0(k)$,
then it holds in the range
 \[ 2 \leq n \leq q-m-k-1. \]
\end{Th}

\begin{Cor} [Case $k=0$.]
  Let $q$ and $m$ be positive integers  composed of the same primes such that the exponent of each prime in $q$ 
  is strictly greater than in $m$, and $m\geq5$. Let $A\subset \setZ_{q} $ be a digital set with  $|A| =m$.
If $A$ is not an arithmetic progression, then $ \xi(n) \geq n+m$ in the range
\[ 2 \leq n \leq q-m-1. \]
\end{Cor}

\begin{Cor} [Case $k=1$.]
  Let $q$ and $m$ be positive integers  composed of the same primes such that the exponent of each prime in $q$ 
  is strictly greater than in $m$, and $m\geq10$. Let $A\subset \setZ_{q} $ be a digital set with  $|A| =m$.
If $\xi(2) \geq m+3$ (that is, $A$ is not a union of at most two arithmetic progressions of a common difference) and
 $\xi(3) \geq m+4$, 
 then $ \xi(n) \geq n+m+1$ in the range
\[ 2 \leq n \leq q-m-2. \]
\end{Cor}

\section{Proof of Theorem \ref{main}.}

We fix the following assumptions:  $q$ and $m$ are positive integers  composed of the same primes 
such that the exponent of each prime in $q$  is strictly greater than in $m$,  $p$ is the smallest prime divisor of $q$,
and $A$ is our digital set with  $|A| =m$.

First we consider adding a subgroup to $A$.

\begin{Lemma} \label{sg}
   Let $H$ be a subgroup of $\setZ_q$, $H \neq \{ \emptyset \}$,  $H \neq \setZ_q$.
   \begin{enumerate}[(i)]
   \item  For every $t$ we have
\begin{equation}\label{sg1}  |A \cap (H+t)| \leq \frac{\min (m, |H| ) }{p}  \leq \frac{\min (m, |H| )  }{2} \end{equation}
 \item For every nonempty subset $A'$ of $A$ we have
\begin{equation}\label{sg2} |A'+H| \geq  p |A'| \geq  2 |A'|  . \end{equation}
\item We have
\begin{equation}\label{sg3}|A+H| \geq ( m|H|, q) \geq
  \begin{cases}
    p \, \max(m, |H| ) \geq    (p-1)m + |H|, \\  \min \left(q, \frac{4}{3}m + |H|. \right)
  \end{cases}
\end{equation}
   \end{enumerate}
\end{Lemma}

\begin{proof}
  Write $|H| =n$. We have $n|q$, $1<n<q$ and
   \[ H = \left\{ 0, \frac{q}{n}, \frac{2q}{n}, \ldots,  \frac{(n-1)q}{n} \right\} . \]
Some of these numbers are congruent modulo $m$, namely, if $m | (jq/n)$, then after $j$ steps the
residues modulo $m$ are repeating. Clearly
 \[ m \Bigm|  \frac{jq}{n} \Longleftrightarrow mn  \Bigm| jq  \Longleftrightarrow \frac{mn}{(mn,q)}   \Bigm| j .\]
Hence
 \[  |A \cap (H+t)| \leq \frac{mn}{(mn,q)}   = 
   \frac{m}{(m,q/n )}  = \frac{n}{(n,q/m )} . \]
 Since both $m$ and $q/m$ contain all prime divisors of $q$, both denominators are divisible by at least one prime factor of $q$,
hence both are $ \geq p$. This shows \kp{sg1}.

To show \kp{sg2}, let let $z$ be the number of cosets of $H$ that intersect $A'$. In each intersection we have
 \[  |A' \cap (H+t)| \leq |A \cap (H+t)| \leq n/p ,\]
so $|A'| \leq zn/p$ while $|A'+H| = zn$.

To prove \kp{sg3}, observe that as any coset of $H$ contains at most $m/(m,q/n)$ elements of $A$, 
hence $A$ must intersect at least $(m,q/n)$ cosets, which together have $n(m,q/n) = (mn,q)$ elements.
Since
\begin{equation}\label{lcm1} (mn, q) = n(m, q/n) \geq pn \end{equation}
 and
  \begin{equation}\label{lcm2} (mn, q) = m(n, q/m) \geq pm, \end{equation}
we immediately get the bound in the upper line. It is stonger than the lower line unless $p=2$.

If $p=2$, then \kp{lcm1} becomes
 \[ (mn, q) = n(m, q/n) \geq 2n, \]
and   \kp{lcm2} can be strengthened to
  \[  (mn, q) = m(n, q/m) \geq 3m ,\]
unless $(n, q/m)=2$. If both inequalities hold, then their arithmetic mean yields the stronger bound $(3/2)m+n$.

If the second inequality fails, then $n$ is a power of 2, say $n=2^j$. If $j=1$, then we have
 \[ (mn, q) = (2m, q) =2m \geq(4/3)m + n = \frac{4}{3}m + 2, \]
as $m\geq 3$. 

 If $j\geq2$, then
$q/m$ must contain 2 exactly in the first power, say
$q=2^sq'$, $m=2^{s-1} m'$ with odd $q', m'$.
 If $q'=m'=1$, then $q|mn$ and $|A+H| = q$. Otherwise $m' \geq 3$, consequently
$m \geq 3 \cdot 2^{s-1} \geq (3/2) n $ and
 \[ (mn, q) = 2m \geq \frac{4}{3}m + n .\]
\end{proof}

\begin{proof}[Proof of Theorem \ref{main}.]

We want to estimate $\xi(n)$ in the range $2 \leq n \leq q-m-k-1$. Let $n$ be the number in this interval where
$\xi(n)-n$ assumes its minimum, and if there are several such values, we take $n$ to be the smallest of them.
Write $\xi(n)-n=m+r$. If $r=k$, we are done, so we suppose that $r \leq k-1$.

 Let $B$ be a set such that $|B| =n$, $|A+B| = m+n+r$. We shall bound $n$ from above in
several stages.

 The set $D = G \setminus (A+B)$ satisfies $|D| = q-(m+n+r)$ and
 \[ (A-D) \subset G \setminus (-B ), \]
 hence $|A-D| \leq q-n = |D| + m+r$. The minimality of $|B|$ implies  $|B|\leq  |D|$, that is,
  \[ n \leq \frac{q-(m+r)}2 . \]
The same argument can be used to show that
 \[ \xi(n)=q \ \text{for } q-m+1 \leq n \leq q \]
(as already mentioned) and
\begin{equation}\label{x1}
  \xi(n)=q-1  \ \text{for } q-m-k-1 \leq n \leq q-m. \end{equation}
This shows that the range $2 \leq n \leq q-m-k-1$ in the Theorem is best possible.

Next we show that $A+B$ is aperiodic. To this end we 
 use \emph{Kneser's theorem:} for any finite sets $A,B$ in a commutative group $G$ we have
 \[ |A+B| \geq |A+H| + |B+H| - |H| ,\]
 where
  \[ H = \{ t\in G: A+B+t=A+B \} ,\]
the group of periodes of $A+B$. If $H= \setZ_q$, then we get $|A+B| \geq |A+H|=q$ and we are done.
If  $H \neq \{ \emptyset \}$,  $H \neq \setZ_q$, then we apply Lemma \ref{sg} to conclude
 \[ |A+H| \geq  \frac{4}{3}m + |H|\]
and so
\[ |A+B| \geq\frac{4}{3}m  + |B+H| \geq\frac{4}{3}m  + |B| \geq m +k+n \]
as wanted (here we use the bound $m \geq 3k$). 

Next we show that $B$ is a Sidon set, that is, for every $t \neq 0$ we have
 $ | B\cap (B+t)| \leq 1$.
 Suppose the contrary. Fix a $t$ such that $ | B\cap (B+t)| \geq 2$ and write
  \[ B_1 = B \cap(B+t), \ B_2 = B \cup (B+t) .  \]
  These sets satisfy
   \[ |B_1| + |B_2| = 2 |B| = 2n, \]
    \[ A+B_1 \subset (A+B) \cap (A+B+t) , \]
 \[ A+B_2 = (A+B) \cup (A+B+t) , \]
consequently 
 \begin{equation}\label{s1} |A+B_1| + |A+B_2| \leq 2 |A+B| = 2(m+n+r). \end{equation} 
$B_1$ must be a proper subset of $B$, since otherwise $B$ and a fortiori $A+B$ would be periodic.
Consequently we have
\begin{equation}\label{s2}
  |A+B_1| > m + |B_1| + r \end{equation}
by the minimality of $|B|$. The set $B_2$ satisfies
 \[ |B_2| =  2n- |B_1| \leq 2n-2 \leq q-(m+r+2). \]
If $ |B_2| \leq q-m-k-1$, then we have
 \begin{equation}\label{s3} |A+B_2| \geq  m + |B_2| + r \end{equation}
 by the definition of $r$. If
  \[    q-m-k-1 <   |B_2| \leq q-(m+r+2),  \]
  then
   \[  |A+B_2| \geq q-1 >  m + |B_2| + r \]
   by \kp{x1}, so \kp{s3} holds anyway. By adding \kp{s2} and \kp{s3} we obtain
    \[ |A+B_1| + |A+B_2| > 2m +|B_1| + |B_2| + 2r =  2(m+n+r), \]
which contradicts \kp{s1}.

Since $B$ is a Sidon set, we have (see \cite{Ruzsalineq})
 \[ |A+B| \geq \frac{mn^2}{m+n-1} .\]
This inequality holds for every set of $m$ elements and it is nearly best in this generality; to use the special properties of $A$ 
we will need another approach.

Comparing this lower bound with the value $m+n+r$ yields the inequality
 \[ mn^2 \leq (m+n+r) (m+n-1) \leq (m+n+k-1) (m+n-1) .\]
 This is a quadratic inequality in $n$ and it gives the bound
  \[ n\leq  \frac{b+\sqrt{b^2 +4ac}}{2a}, \ a=m-1, \ b=2m+k-2, \ c = (m-1)(m+k-1). \]
    For large $m$ this is asymptotic to $\sqrt{m}$; in particular, there is an $m_0$ depending on $k$ such that
     \[ \beta = \frac{|A+B|}{|A|} = \frac{m+n+r}{m} < \sqrt{2} \]
for $m>m_0$.  Such a bound is easily found in the particular cases $k=0,1$;
if $k=0$, it holds for $m \geq 5$, if $k=1$, it holds for $m\geq 10$.

Pl\"unnecke's theorem (see \cite{Ruzsapluennecke}) implies the existence of a nonempty subset $A'$ of $A$ such that
 \begin{equation}\label{b2} |A'+ 2B | \leq \beta^2 |A'| < 2 |A'| . \end{equation}
 We shall compare this to the Kneser bound
  \[ |A'+ 2B | \geq |A'+H| + |2B+H| - |H|, \]
  where $H$ is the group of periodes of  $A'+ 2B$. If $H$ is a nontrivial subgroup, then
   \[ |A'+H| \geq 2 |A'| \]
by \kp{sg2}; this also holds trivially if $H=\setZ_q$, and this contradicts \kp{b2}. 

If $H=\{ \emptyset \}$, then Kneser's bound reduces to
 \[ |A'+ 2B | \geq |A'| + |2B| - 1 = |A'| + \frac{n(n+1)}2 -1 , \]
 as $|2B| = n(n+1)/2$ by the Sidon property. A comparison with the upper estimate \kp{b2} gives
 \[ |A'|+ \frac{n(n+1)}2 -1 \leq \left(\frac{m+n+r}{m} \right)^2 |A'|,   \]
 \[ \frac{n(n+1)}2 -1 \leq  |A'|\left( \left(\frac{m+n+r}{m} \right)^2 -1 \right)   \]
  \[  \leq  m \left( \left(\frac{m+n+r}{m} \right)^2 -1 \right) = \frac{(2m+n+r)(n+r)}{m} \leq   \frac{(2m+n+k-1)(n+k-1)}{m}. \]
This is again a quadratic inequality in $n$ and it gives the bound
  \[ n\leq  \frac{b+\sqrt{b^2 +4ac}}{2a}, \ a=m-2, \ b=3m+4k-4, \ c =2m+ 2(k-1)(2m+k-1). \]
As $m \to \infty$, this bound tends to $\left(3+\sqrt{16k+1}\right)/{2}$. The bound $m_0$ after which we can claim this
bound for $n$ depends on the fractional part of the square root inside, but it is easily found in the particular cases $k=0,1$;
if $k=0$, it holds for $m \geq 4$, if $k=1$, it holds for $m\geq 9$.
\end{proof}

\begin{section}{Arithmetic progression structure of sets}

We are interested in studying sets $A \subseteq \Z_q$ containing no nontrivial cosets such that $\xi_A(2)=\xi_A(3).$

If such an equality were to hold, then there exist non zero elements $d_1\neq d_2$ such that
$$|A| + k = \xi_A(3) = |A + \{ 0, d_1, d_2 \}| \geq |A+ \{ 0, d_1 \}| \geq \xi_A(2) = \xi_A(3),$$
so that $|A+\{ 0, d_1, d_2\}| = |A+ \{0, d_1\}| =|A+ \{0, d_2\}|=|A+ \{d_1, d_2\}|$. 
In particular, this tells us that the set $A$ can be written as the union of $k$ arithmetic progressions of difference $d_1$ or $d_2$, and there exists 3 distinct elements $x_i \in \Z_q$ such that $$  \bigcup_{i=1}^3 (A+x_i) = (A+x_a )\cup (A+x_b) $$ for any choice of distinct $a, b \in \{1,2,3\}$.

In the sequel we will study these two problems.

Given integers $a, b$ we let the interval $[a,b] \subseteq \Z_q$ be the image of $[a, b']$ under the natural projection $\varphi: \Z \to \Z_q$, where $b'$ is the minimal integer such that $b' \equiv b$ mod $q$ and $a \leq b'$.

 Let $|x| = \min\{|x+kq| : k \in \Z\}$ for $x\in \Z_q$ be the seminorm measuring the distance of an element in $\Z_q$ from zero.

If $\xi_A(2) =|A|+k$, is the decomposition of $A$ as the union of $k$ arithmetic progression unique up to a sign?

In other words, can there be two proper decomposition of $A$ as a
$$A=\cup_{i=1}^k P_i = \cup_{i=1}^k Q_i,$$ 
$$P_i= \{a_i, a_i+d_1, \dots, a_i + k_i d_1 \}, \quad Q_i = \{a'_i, a'_i + d_2, \dots, a'_i + k_i d_2 \}$$
 with $d_1 \neq \pm d_2$, $d_1, d_2 \in (-q/2, q/2]$?


If $A$ is an arithmetic progression of difference $d$ itself, so that $k=1$, the only possibility is clearly $d_1= \pm d_2$.

Suppose now $k=2$. Very small (or, by taking their complement, very large) sets $A$ with $|A| \leq 4$ may have multiple representation as union of two arithmetic progression, as happens for sets of the form $A=\{ a, a+d, b, b+d\}$.

On the other hand, we can easily provide examples of different minimal arithmetic progression decompositions if the ratio $|d_1/d_2|$ or $|d_2/d_1|$ is less or equal to $2$, as happens for sets of the form $A=[a, b] \cup \{ b+2\}$ or $A=\{a-2\} \cup [a, b]$

The following Theorem states that these are the only kinds of sets having multiple minimal arithmetic progression decompositions.


\bteo\label{k2mv1}
Let $A \subseteq \Z_q$, $4 < |A| < q-4$. Assume that $q$ is odd, $q>100$ and $A$ is not contained in a coset of any nontrivial subgroup of $\Z_q$.
If $\xi_A(2)=|A+\{ 0, d \}| =|A|+2$, then the only elements  $x \in \Z_q$ with $|A+\{0, x\}| = |A|+2$ are $\pm d$, unless $A$ is a dilation of sets of the form $[a, b] \cup \{ b+2\}$ or $\{a-2\} \cup [a, b]$ for suitable $a, b \in \Z_q$.
\eteo

\bdimo

{\bfseries Case 1}: $(d_1, q)=(d_2, q)=1$.

Let $A$ be a set with $\xi_A(2)=|A|+2$ having a double decomposition as the union of two proper arithmetic progression of difference $d_1$ or $d_2$. Dilating $A$ by $d_2^{-1}$ we can assume that $A$ is the union of two disjoint intervals in $\Z_q$. Also, by taking the complementary of $A$, we can assume $|A| < q/2$. (This may fail if the differences are not coprime to $q$; then possibly the complement is the union of the same number of arithmetic progressions and some cosets of the subgroup generated by the difference.)

Let $A=I_1 \cup I_2 = P_1 \cup P_2$ where $P_i$ are arithmetic progressions with common difference $1<d<q/2$ and $I_i = [a_i, b_i]$.

Let $d=\frac{q+1}{2}-x$ for a positive integer $x<\frac{q-1}{2}$. Either $d^{-1}$ or $-d^{-1}$ must be conguent to $\frac{q+1}{2}-y$ for a positive integer $y<\frac{q-1}{2}$. Then 
$$\pm4 \equiv 4d(\pm d^{-1}) \equiv (2x-1)(2y-1)\mbox{\quad mod $q$}, $$ which implies that either $x$ or $y$ must be greater than $\frac{\sqrt{q-4}+1}{2} \geq \frac{\sqrt{q}}{2}$.

Hence we can also assume $1< d<(q-\sqrt{q})/2$.

We say that a progression $P_i = \{ a+kd : k =0, \dots, l \}$ jumps from $I_1$ to $I_2$ at $l$ if $a+(l-1)d \in I_1 \cap P_i$ and $a+ld \in I_2 \cap P_i$.

We now split the proof into two cases.

{\bfseries Subcase 1}: $d=2$.

Since $|A|<q/2$, neither $P_1$ nor $P_2$ can jump from $I_1$ to $I_2$ or viceversa more than once. Then it's easy to see that the only possibility is that $A$ behaves as in the statement of the theorem.

{\bfseries Subcase 2:} $d>2$.

Since $A < q/2$ there must be a gap between the intervals $I_1$ and $I_2$ of length $g>q/4$.
 Let $|I_1| \leq |I_2|$ and $a_1 - b_2 - 1 \equiv g$ mod $q$, so that $|I_2| > 2$ and hence $I_2$ contains $3$ consecutive elements. Then at least one of the $P_i$'s must jump from $I_2$ to $I_1$ and then to $I_2$ again, implying that $d > g >q/4 > |I_1|$, and that at least one element $x'$ in $I_1$ satisfies $ x' \pm d \in A$.

There are at most 4 elements $x \in A$, the starting and ending points of the $P_i$'s, such that $\{ x+d, x-d \} \not\subseteq A$. So we can find an element $y \in [a_1, a_1+4] \cap A \subseteq I_1$ such that $y \pm d \in A$, either by taking $y=x'$ if $|I_1| < 5$ or $y$ as a point in the middle of an arithmetic progression if $|I_1| \geq 5$.

Since $|I_1| < d$ we have $y \pm d \in I_2$, and so the interval $[y+d, y-d]$ must be contained in $I_2$.

Take now an element $z \in [y-d-7, y-d-5] \subseteq [y+d, y-d]$ which is not the ending element of $P_1$ or $P_2$, so that $z + d \in A$, to obtain a contradiction since $z+d \in [y-7, y-5] \subseteq [a_1 -7, a_1 - 1]\subseteq A^c$. (Here we need that $2d+7 \leq  q$, which follows from the assumption on the size of $q$ and the above inequality for $d$.)
\edimo

To proceed to the case of not coprime differences we need a simple lemma which allows us to normalize the differences of the arithmetic progressions.

\blem\label{redd2}
For arbitrary integers $a, q$ there exists an integer $a'$, $a' \equiv a$ mod $q$ and $a' = a_1 a_2$, with $a_1 | q$ and $(a_2, q)=1$
\elem
\bdimo
Let $I = \{ p : p \mbox{ prime}, v_p(a) = v_p(q)\}$.

Define $a':= q\prod_{p \in I} p + a$. Then $a' \equiv a$ mod $q$, $v_p(a') = v_p(q)$ for all primes $p \in I$ and $v_p(a') = \min(v_p(a), v_p(q)) \leq v_p(q)$ for all primes $p \not\in I$.
\end{proof}


\bdimo

Let $q=\prod p_i^{r_i}$ be the decomposition of $q$ as a product of powers of distinct primes.

Let $A=P_1 \cup P_2 = Q_1 \cup Q_2$, with $P_i$'s arithmetic progressions of difference $d_1$ and $Q_i$'s of difference $d_2$, with $P_i = \{\alpha_i, \alpha_i + d, \dots \}$.



{\bfseries Case 2}: $(d_1,q) = 1 < (d_2, q)$.

After a dilation we can assume $d_1=1$, and so there are three consecutive elements $\{\gamma, \gamma +1, \gamma+2 \}$ contained in $A$.

However, since $2\nmid q$, we have that $d=(d_2, q) > 2$ and so the union of $Q_1$ and $Q_2$ can cover at most two of these three elements, which is a contradiction.

{\bfseries Case 3}: $(d_1, q), (d_2, q) > 1$.

After a dilation, thanks to Lemma \ref{redd2}, we can assume $d_1 | q$.

If $\alpha_1 \equiv \alpha_2$ mod $d_1$ then $A$ is contained in a single coset of the subgroup generated by $d_1$, contrary to the assumption.

If $\alpha_1 \not\equiv \alpha_2$ mod $d_1$ then $P_i = \{x \in A : x\equiv \alpha_i \mbox{ mod $d_1$}\}$.

If $d_1 | d_2$ then we also get $Q_i= \{x \in A : x\equiv \alpha_{\varphi(i)} \mbox{ mod $d_1$}\}$ for a permutation $\varphi:\{1,2\} \to \{1,2\}$, and the result follows immediately.

If $d_1 \nmid d_2$ then, letting $Q_1=\{q_1, q_2, q_3, \dots \}$ be an arithmetic progression with at least three elements, we have $q_1 + d_2 \not\equiv q_1$ mod $d_1$ and so $q_1 + 2d_2 \equiv q_1$ mod $d_1$, which implies that $2|q$,  again a contradiction.
\edimo

Trying to prove results similar to Theorem \ref{k2mv1} for higher $k$ becomes a harder task, since new families of exceptions have to be considered, as  already shown for $k=3$ by the  set
$$A=[1,a] \cup \left(\left[(p-1)/2,(p-1)/2+a-1\right] \setminus \left\{ 1+(p+1)/2 \right\}\right),$$ which is the union of 3 intervals as well as of 3 arithmetic progressions of difference $d=1+\frac{p+1}{2}$.

For $k>2$ we also still find the same families of sets having more than one decomposition which we found for $k=2$: sets $A$ with $|A| \leq k^2$ or with $|d_1/d_2| \leq k$. 

In the former case, $|A| \leq k^2$, there exists an arithmetic progression in its decomposition having cardinality less or equal than $k$, so after removing its points from $A$ we obtain a set $\widetilde{A}$ with $|\widetilde{A}| \geq |A|-k$ and $\xi_{\widetilde{A}}(2) - |\widetilde{A}| \leq k-1$.

In the latter case, $|d_1/d_2| \leq k$, after multiplying the set $A$ by $\pm d_2^{-1}$, we have that $A=I_1 \cup \dots \cup I_k = P_1\cup \dots \cup P_k$ for intervals $I_i$'s and arithmetic progressions $P_i$'s of difference $d\leq k$. 
Since at least one of these arithmetic progressions must jump from one interval to another there exists a gap between two intervals of length less or equal than $k$, and so, by adding those points to $A$ we obtain a set $\widetilde{A}$ with $|\widetilde{A}| \leq |A|+k$ and $\xi_{\widetilde{A}}(2) - |\widetilde{A}| \leq k-1$.

The common point between these two kinds of sets and the multitude of other types of examples one can produce as $k$ grows, is that even though they both are the union of $k$ $d$-arithmetic progression, they are actually obtained by sets $\widetilde{A}$ which are the union of $k-1$ $d$-arithmetic progressions by removing or adding up to $k$ elements.

To exclude these sets, we give the following definition.

\begin{Def}\label{def1}
$A$ has $k$ stable components if $\xi_A(2) = |A| + k$, and for any $d$ such that $|(A+d) \setminus A|=k$, any set $\widetilde{A}$ obtained by $A$ by removing or adding up to $k$ elements satisfies $|(\widetilde{A}+d) \setminus \widetilde{A}| \geq k$.
\end{Def}

Moreover, if we work in the composite number modulus case, new sets having multiple representation as union of a minimal number of arithmetic progressions can be found, because of the presence of nontrivial cosets in this setting.

Of course, the union of $k$ disjoint cosets has a lot of representations as the union of $k$ arithmetic progressions, but it is not hard to find other less trivial sets which satisfy this property.

For example, for suitable $k,q$, $k\mid q$, 
\begin{equation}\label{mcontro}A = [0,2k-1] \bigcup_{i=1} ^{k-1} [q/k+i, q/k+(k+1)+i]\subseteq \Z_q
\end{equation}
 is a set of $k$ stable components which is not the union of cosets but still is the union of either $k$ intervals or $k$ arithmetic progressions of difference $d=q/k +1$.

Nevertheless, this set $A$ has high density in some coset of $\Z_q$, namely $\langle q/k \rangle$.

In the following Theorem we show that the essential uniqueness of the decomposition of a set into $k$ arithmetic progressions still holds for sets of $k$ stable components and with low density into any coset of $\Z_q$.

Moreover, it will be clear from the proof that the only sets of $k$ stable components with such multiple decompositions will be of the same kind as the set in \eqref{mcontro}.
\bteo
Let $A \subseteq \Z_q$ be the union of $k$ arithmetic progressions of difference $d_1$ and $d_2$, $|A\cap(H+t)|<|H|/2$ for any nonzero coset $H+t$ of $\Z_q$ , and $A$ has $k$ stabel components.
Then $d_1 = \pm d_2$.
\eteo
\bdimo
Since we are going to prove $d_1 = \pm d_2$, and every arithmetic progression of difference $d$ is also an arithmetic progression of difference $-d$, during the course of the proof we are going to choose suitable signs for $d_i$ in order to simplify the notations.

Let $A=P_1 \cup \dots \cup P_k = Q_1 \cup \dots \cup Q_k$ with $P_i$'s being arithmetic progressions of difference $d_1$ and $Q_i$'s of difference $d_2$.

We denote by $S_i$ and $E_i$, $i=1, 2$ the starting and ending points of the arithmetic progressions of difference $d_i$ forming $A$, i.e.
$$S_i = \{ x \in A: x-d_i \not\in A\}, \qquad E_i= \{ x \in A : x+d_i \not\in A\},$$
with $|S_i|=|E_i|= k$.

Given $x, y$, we will write $x \sim_i y$ for $i=1, 2$ if $x, y \in A$ and they both belong to the same arithmetic progression of difference $d_i$.

Since $A$ has $k$ stable components, the following properties hold:
   \begin{enumerate}[(i)]
\item\label{pr1} $|P_i|, |Q_i| \geq k+1$ $\forall i=1,\dots,k$, for if otherwise, by removing a short arithmetic progression, we would obtain a contradiction with Definition \ref{def1}.

\item\label{pr2} If $P_i = \{ a + l d_1: l=0, \dots, M_i-1\}, P_j = \{ a + (M_i+ l) d_1 : l = N, \dots, N+ M_j\}, N >0,$ are two different components contained in the same coset $a + \left< d_1 \right>$, then $N \geq k+1$, for otherwise, by adding the elements $\{ a + ld_1 : l= M_i, \dots, M_i +N-1\}$ to $A$ we would obtain a contradiction with Definition \ref{def1}.
A similar statement holds for $Q_i, Q_j$ and $d_2$ instead of $P_i, P_j$ and $d_1$.

\item\label{pr3} $\forall i \exists j : (P_i +  d_2) \cap A \subseteq P_j$. In fact, if  $P_i \subseteq a + \left< d_1 \right>$ and $(P_i +d_2) \cap A \cap P_{k_l} \neq \emptyset$ for two different components $P_{k_1}$ and $P_{k_2}$, then we have $P_i + d_2 \subseteq a + d_2 +  \left< d_1 \right>$, which implies that both $P_{k_1}$ and $P_{k_2}$ are contained in the same coset of $\left< d_1 \right>$. Then, because of \eqref{pr2}, the set $P_i + d_2$ contains at least $k+1$ elements not belonging to $A$, and hence $|E_2| \geq k+1$, a contradiction.
A similar statement hold for $Q_i$ and $d_1$ instead of $P_i$ and $d_2$.

\item\label{pr4} $\forall i \exists j : (P_i -  d_2) \cap A \subseteq P_j$ and $\forall i \exists j : (Q_i -  d_1) \cap A \subseteq Q_j$, by an argument similar to \eqref{pr3}.
\end{enumerate}

Thanks to Lemma \ref{redd2} we can assume, after a dilation, $d_1, d_2 \in [0, q-1]$, $d_2 \mid q$.

Let $d=(d_1, d_2)$, $d_i = d_i' d$ for $d=1, 2$, $q = dq'$ and $\mathcal{A}_i = \{x \in A : x \equiv i \mbox{ mod $d$}\}$.

Clearly, if $P_j \cap \mathcal{A}_i \neq \emptyset$, then $P_j \subseteq \mathcal{A}_i$, and the same holds for the $Q_j$'s, so that every $\mathcal{A}_i$ is the union of $r_{1, i}$ $d_1$-arithmetic progressions and $r_{2,i}$ $d_2$-arithmetic progressions.

We are going to show that the ratio $r_{1,i}/r_{2,i}$ is constant for every $i$.

Let $A_i = \frac{\mathcal{A}_i - i}{d} \subseteq \Z_{q'}$.

Clearly every set $A_i$ inherits from $A$ the same stability properties (relative to $k$) and the condition of density into cosets.

We use the same notation above for subsets of $\Z_{q'}$.

\bclaim $ r_{2, i} \geq d_2'$.
\eclaim
\bdimo[Proof of claim]
Since $d_2'|q'$, $x\sim_2 y$ implies $x \equiv y$ mod $d_2'$.

Given $s \in S_1$, if by contradiction  $m' >d_2' > r_{2,i}$ then the set $B=\{ s, s+ d_1', \dots, s + r_{2,i} d_1' \} \subseteq A_i$ has cardinality $r_{2,i}+1$ since $r_{2,i} \leq k$.

For $j \in [0, r_{2,i}] \subseteq [0, d_2'-1]$, $jd_1' \equiv 0$ mod $d_2'$ can only happen for $j=0$ by the coprimality of $d_1'$ and $d_2'$.

Hence $B$ intersects $r_{2,i}+1$ distinct $d_2'$-arithmetic progression, which leads to the contradiction.
\edimo

Let $X=\{ x \in [k]: xd'_1 \equiv 0 \mbox{ mod $d'_2$}\}$.

From $d'_2 \leq r_{2,i}\leq k$ we get $X \neq \emptyset$.

For every $x \in X$ let $\beta_+(x)$ be the minimal positive integer such that $xd'_1 \equiv \beta_+(x)d'_2$ mod $q'$, and  $\beta_-(x)$ be the minimal positive integer such that $-xd'_1 \equiv \beta_-(x)d'_2$ mod $q'$. Let $\beta(x) = \min (\beta_+(x), \beta_-(x))$ and $\beta(\alpha) = \min_{x\in X} \beta(x), \alpha \in [k]$.

After changing $d'_1$ with $-d'_1$ we can assume $\beta(\alpha) = \beta_+(\alpha)$.

Let $S_1 = \{ s_1, \dots, s_{r_{2,i}}\}$. For every $i$ define $l_i$ to be the minimal integer such that $s_i + l_id'_1 \sim_2 s_i$. Clearly $l_i \in X$, and one of the following must happen, according to which one between $\beta_+(l_i)$ and $\beta_-(l_i)$ is minimal:

   \begin{enumerate}[(i)]
\item $s_i + ld'_2 \in A_i$ for $l \in [0, \beta_+(l_i)]$ 
\item $s_i - ld'_2 \in A_i$ for $l \in [0, \beta_-(l_i)]$ 
\end{enumerate}

\boss\label{oss123}

Because of \eqref{pr3} all the elements $x$ such that $x \sim_{1,2} s_i$ are of the form $s_i + l l_i d'_1$ for some $l \geq 0$. In particular, $\beta(\frac{x-s_i}{d'_1}) > \beta(l_i)$, otherwise $|A_i \cap (s_i  + \langle d'_2\rangle)| > \frac{|\langle d'_2\rangle|}{2}$.

Moreover, all those elements $x$ belong to the same semicircle $[s_i, s_i + m'/2)$ or $(s_i - m'/2, s_i]$.
\eoss

Suppose $\beta(l_i) > k$ for all $i$. Suppose $\beta(l_1) = \beta_+(l_1)$ and $\beta(l_1) = \min_{i=1, \dots, r_{2,i}} (\beta(l_i))$. Then the set $\{ s_1, s_1  + d'_2, \dots, s_1 + l_1 d'_1 = s_1 + \beta(l_1)d'_2 \} \subseteq A_i$ intersects at least $k+1$ different $d'_1$-arithmetic progression, leading to a contradiction. 
A similar argument works if $\beta(l_1) = \beta_-(l_1)$).

But then, since $\beta(l_1), l_1 \leq k$, we get that for every $j$, $s_j + l_1 d'_1 = s_j + \beta(l_1) d'_2 \sim_{1,2} s_j$. Moreover, Remark \ref{oss123} tells us that $l_1=l_j = \alpha$ for all $j$. 

Split the set $A_i$ into $M$ equivalence classes under the relation $P_{j_1} \sim P_{j_2}$ if there are $p_1 \in P_{j_1}$, $p_2 \in P_{j_2}$, with $p_1 \sim_2 p_2$. This is well defined by \eqref{pr3}.

Each equivalence class is composed by $\alpha$ $d'_2$-arithmetic progressions, so that $r_{2,i} = M \alpha$.

If $x, x+ \alpha d'_1 \in A_i$, there does not exists a $y \in \{ x+ l d'_2, l \in [0, \beta(\alpha)) \}$ with $y \sim_1 x$, and hence $k \geq r_{1,i} \geq \beta(\alpha)$. On the other hand, we already know that $x  \sim_1 x+ \alpha d'_1$, and so $r_{1,i} = M \beta(\alpha)$.

In particular, the ratio $r_{1,i}/r_{2,i} = \beta(\alpha) / \alpha$, a constant not depending on $i$.

Since $A$ is the union of $k$ $d_1$-arithmetic progressions and $k$ $d_2$-arithmetic progressions, we must have $\beta(\alpha)=\alpha$.

We now show that this leads to $d'_1 = d'_2$, which concludes the proof since, after dilating the set $A$ so that $d_2 \mid q$, we did already choose between $d_1$ and $-d_1$ in order to simplify the notation.

Going back to $A_i$ we have $\alpha d'_1 \equiv \alpha d'_2$ mod $q'$, and so, for $D=(\alpha, q')$ we get $\frac{q'}{D} \mid \frac{\alpha}{D}(d'_1 - d'_2)$ and so $d'_1 = d'_2 + j \frac{q'}{D}$ for some $j\geq 0$.

Assume by contradiction that $D> j > 0$.

We already know that $B=\{s_1, s_1+ d'_2, \dots, s_1 + \alpha d'_2 = s_1 + \alpha d'_1 \} \subseteq A_i$.

Let $D'$ be the additive order of $j\frac{q'}{D}$ in $\Z_{q'}$, $D' \leq  D \leq \alpha \leq k$.

Then $s_1 + D' d'_1 = s_1 + D' d'_2 \in B$ and $s_1 + D' d'_1 \sim_2 s_1$, so that $D' = \alpha$

Moreover, $s_1 + l d'_1 \in A_i$ for $0\leq l \leq \alpha$.

By \eqref{pr1} and $\alpha \leq k$ we have that at least one between 
$$ld'_1 - ld'_2 = l j \frac{q'}{D} \quad \mbox{ or }\quad ld'_1 + (\alpha - l) d'_2 = \alpha d'_2 + l j\frac{q}{D}$$ belong in $A_i$, and so at least one of the two cosets $\langle j\frac{q'}{D}\rangle$ and $\alpha + \langle j\frac{q'}{D}\rangle$, both having cardinality $\alpha$, intersects $A_i$ in more than half of its elements, which leads to a contradiction with our hypothesis of low density in cosets.

Hence $j=0$ and $d'_1 = d'_2$.

\edimo
\end{section}

\begin{section}{Sets $A$ with $\xi_A(2)=\xi_A(3)$}
Let $A \subseteq \Z_q$ be a set which does not contain any non trivial cosets, with $|A|=m$, $\xi_A(2)=\xi_A(3)$. Then there are $d_1 \neq d_2$ such that

\begin{equation}\label{eqmodq}A+ \{0,d_1,d_2\}=A+ \{0,d_1\}=A+ \{0,d_2\}=A+\{d_1,d_2\}, \end{equation}

After a dilation, applying Lemma \ref{redd2}, we can assume $d_1, d_2 \in [0, q-1]$ and $d_1 | q$. Let $H = \langle d_1 \rangle$ be the subgroup generated by $d_1$, so that $|H|=q/d_1$.

As usual, write $A=P_1 \cup \dots \cup P_k = Q_1 \cup \dots \cup Q_k$ as the union of $k$ $d_1$-arithmetic progressions $P_i$'s as well as $k$ $d_2$-arithmetic progressions $Q_i$'s, with 
$$P_i = \{ a_i + j d_1; j = 0, \dots, j_i \}, \quad a_i + j_i d_1= b_i,$$
$$Q_i = \{ \alpha_i + l d_2; l = 0, \dots, l_i \}, \quad \alpha_i + l_i d_2= \beta_i$$

Since
$$A+ \{ 0,d_1\} = A \amalg \{ b_i + d_1\}_{i=1, \dots, k} $$
$$A+ \{ 0,d_2\} = A \amalg \{ \beta_i + d_2\}_{i=1, \dots, k} $$

we have
\begin{equation}\label{bbeta}\{ b_i + d_1\}_{i=1, \dots, k}  = \{ \beta_i + d_2\}_{i=1, \dots, k}.\end{equation}

Suppose that set $A$ has non empty intersection with $z$ cosets of $H$.

Let $\{G_i\}_{i=1, \dots, k}$ be the set of maximal $d_1$-arithmetic progressions contained in those $z$ cosets of $H$ such that $G_i \subseteq A^c$.
In particular, after a reordering, we can assume $G_i = \{x_i + h d_1, h = 0, \dots, h_i\}$, with $x_i = b_i + d_1$ and $x_i + h_i d_1 = a_{\varphi(i)}-d_1$ for a permutation $\varphi: [k] \to [k]$.

Note that $a_i \in A+\{d_1, d_2\} \setminus A+d_1$, for otherwise $A$ would contain a full coset of $H$. 

Hence 
\begin{equation}\label{aia}a_i - d_2 \in A\end{equation}
and from \eqref{bbeta} and \eqref{aia} we deduce that
$$(G_i - d_2) \cap A = \{ \beta_{\tau(i)}\}$$
for a permutation $\tau: [k] \to [k]$. Moreover, either $|G_i|=1$ or $(G_i - d_2) \cap A^c = G_j$ for another $G_j$ with $|G_j| = |G_i|-1$.

We can then define a partial order $\leq$ on the $G_i$'s by $G_a \leq G_b$ if and only if $\exists i \geq 0$ such that
$$G_a = (G_b - i d_2) \cap G_b - i(d_2-d_1).$$

A $G_i$ which is maximal for this partial order satisfies $G_i + d_2 \subseteq \{ \alpha_i\}_{i=1, \dots, k} \subseteq A$, and so $|G_i| \leq k$, leading to
\begin{equation}\label{eqmodq}|A| \geq z|H| - \frac{k(k+1)}{2}.\end{equation}

We have then proved the following:

\bteo\label{teo01dq}
Let $A \subseteq \Z_q$  be a set not containing any nontrivial cosets and which satisfies 
$$\xi_A(2)=\xi_A(3).$$
Then there exists a $d_1|q$ such that $A$ intersects $z$ cosets of $H=\langle d_1 \rangle$ and, after a dilation, $A$ is of the form 
$$\Z_q \setminus \left(\coprod_i \mathcal{G}_i \coprod   \coprod_{j=1}^{d_1-z} (t_j + H )\right), $$
 where $\mathcal{G}_i$ are  chains $\mathcal{G}_i = \{ \{ g_i\} = G_{i, 1} \leq \dots \leq G_{i, j_i}\}$ with 
   \begin{enumerate}[(i)]
\item\label{cco1} $|G_{i, j_i}| \leq \xi_A(3) - |A|$, 
\item $|G_{i, j-1}| = |G_{i, j}| -1$,
\item $g_i -d_2 \in A$, 
\item\label{cco4} $(G_{x,y }+\{0,d_1\}) \cap (G_{w,z} + \{0, d_1\}) = \emptyset$ for $(x,y)\neq(w,z)$.

\end{enumerate}
\eteo

Restricting ourselves to the case $q=p$ prime, it is an interesting question to study the minimal cardinality of $A$ in order to have $\xi_A(2)=\xi_A(3)$.

A rectification argument (see \cite{BiluLevRuzsa} and \cite{Levrect}) shows that $|A| > \log_4(p)$.
Since every element $a \in A+\{ 0, d_1, d_2\}$ must belong to at least two sets $A+x$, $x\in \{0, d_1, d_2\}$, as long as $|A| < 2/3 p$ we have 
$$k = |A+\{0, d_1, d_2\}|-|A| \leq \frac{|A|}{2}.$$

This, combined with the bound in \eqref{eqmodq}, gives
$$|A| \geq \sqrt{8p +25}-5. $$

Let $\mu(p) = \min (|A|: A \subseteq \Z_p \mbox{ and $A$ satisfies $\xi_A(2)=\xi_A(3)$ } )$. We conjecture the following:

\bconj\label{conj}
$$\lim_{p\to \infty}\frac{\mu(p)}{p} > 0.$$
\econj

In the following we will show that $\liminf_{p\to \infty}\frac{\mu(p)}{p} \leq \frac{5}{18}$.

To do this we construct sets $B \subseteq [0, 2^{2m}]$ of cardinality $|B| = \frac{13}{18}2^{2m}+ o(2^{2m})$ which is the union of disjoint chains satisfying conditions \eqref{cco1}-\eqref{cco4} in Theorem \ref{teo01dq}.

Since by \cite{Pintz} there exists a prime $p$ in $[2^{2m}, 2^{2m}+2^{21m/20}]$, the complement of the image of the canonical projection of $B$ into $\Z_p$ will have density asymptotic to $5/18$ as required.

Let $d=2^m$, $\mathcal{G}_l = \{ \{0\} \leq [d-1, d] \leq \dots \leq [d(l-1) - (l-1) , d(l-1)] \}$ for $l\leq d$ and $H_i = (id-d, id]$. Let $\varphi(\mathcal{G}_l) = d+ \mathcal{G}_{l-1}$ be the chain of intervals obtained from $\mathcal{G}_l$ by removing the first element in each of its intervals.

If $C= \cup_{i \in I} \mathcal{G}_{l_i} + x_i$ and $\mathcal{G}_a \cap \mathcal{G}_b = \emptyset$ for all $a,b \in I$, then the set $B=\cup_{i \in I} \varphi(\mathcal{G}_{l_i})+x_i$ satisfies the conditions of Theorem \ref{teo01dq}.

Let
$$C=C_0 \coprod_{l=1}^{m-1} \coprod_{i=1}^{m-l} B_i^{(l)},$$
where
\begin{eqnarray*}
C_0&=& \mathcal{G}_{2^m},\\
B_i^{(l)} &=& 2^m(2^{m+1-l} -2^{m+2-l-i} - 1) + 2^{m+1-l-i} + \mathcal{G}_{2^{m+1-l-i}}.
\end{eqnarray*}

If we denote by $B_{i,k}^{(l)}$ the $k$-th interval of the chain, $0 \leq k \leq 2^{m+1-l-i}-1$, we have that
\begin{multline*}B_{i,k}^{l} = [2^m(2^{m+1-l} -2^{m+2-l-i} -1+k) + 2^{m+1-l-i} -k,\\ 2^{m}(2^{m+1-l} -2^{m+2-l-i} -1+k) + 2^{m+1-l-i} ]
\end{multline*}

Suppose now that $B_{i,k}^{(l)} \cap B_{i',k'}^{(l')} \neq \emptyset$. Then, since $B_{i,k}^{(l)} \subseteq H_{2^{m+1-l} -2^{m+2-l-i} +k}$, for $\alpha(l, i, k) = 2^{m+1-l} -2^{m+2-l-i} +k$, we must have $\alpha(l, i, k) = \alpha(l', i', k')$.

{\bfseries Case 1:} $i, i' \geq 2$.

In this case we have $\alpha(l, i, k) \in [2^{m-l}, 2^{m+1-l})$ and since any two of these intervals are disjoint, we must have $l = l'$, which implies that $$k-2^{m+2-l-i} = k'-2^{m+2-l'-i'} \in [-2^{m+2-l-i'}, -2^{m+1-l-i'}].$$ Again, since any two of these intervals are disjoint, we must have $i=i'$, which immediately gives $k=k'$.

{\bfseries Case 2:} $i=1$.

In this case from the equality $\alpha(l, i, k)=\alpha(l', i', k')$ we have $$k = 2^{m+1-l'}-2^{m+2-l'-i'}+k'.$$

If $i' \geq 2$, then the left hand side is in $[0, 2^{m-l})$, while the right hand side belongs to $[2^{m-l'}, 2^{m+1-l'})$. From this we get that $m-l' < m-l$ and so $\max(B_{i,k}^{l}) > \max{B_{i', k'}^{l'}}$.
Moreover, we have
$$2^{m-l}-k = 2^{m-l} - 2^{m+1-l'} + 2^{m+2-l'-i'} - k' > 2^{m+1-l'-i'}$$
since $k' < 2^{m+1-l'-i'}$, so that $\max{B_{i', k'}^{l'}} < \min(B_{i,k}^{l})$ and $B_{i,k}^{(l)} \cap B_{i,k}^{(l')} = \emptyset$.

If  also $i'=1$, then $k=k'$ and, if $l < l'$, we have $k \leq 2^{m-l'}-1 \leq 2^{m-l-1} -1$, so that $2^{m-l}-k \geq 2^{m-l-1} +1 \geq 2^{m-l'} +1$, and so  $B_{i,k}^{(l)} \cap B_{i,k}^{(l')} = \emptyset$.

Since $|\varphi(\mathcal{G}_l)| = \frac{l(l-1)}{2}$, for $B=\varphi(C_0) \coprod_{l=1}^{m-1} \coprod_{i=1}^{m-l} \varphi(B_i^{(l)})$, we have

\begin{eqnarray*}
|B| &=& \frac{2^m(2^m-1)}{2} + \sum_{l=1}^{m-1} \sum_{i=1}^{m-l} \frac{2^{m+1-l-i}(2^{m+1-l-i}-1)}{2} = \frac{13}{18}2^{2m} + o(2^{2m})
\end{eqnarray*}
as required.

Go back to the general case of composite modulus $q$.
An analogue of Conjecture \ref{conj} cannot hold in this case, as we can just take a set $A' \subseteq \Z_{q'}$ with $\xi_{A'}(2)=\xi_{A'}(3)$ and consider the set $A = A' \times \{0\} \subseteq \Z_{q'} \times \Z_{q''}=\Z_q$ for any coprime $q', q''$ with $q=q'q''$.

We can now finish the proof of Theorem \ref{digsetteo}and Corollary \ref{digsetcorollary}.

If $A$ is a digital set with $\xi_A(2) = m+3 = \xi_A(3) = |A+\{0, d_1, d_2\}|$ as in Theorem \ref{teo01dq}, we have by Lemma \ref{sg} and \eqref{eqmodq} that 
$$|A| \geq z|H| -6,$$
with $z \in \{2,3 \}$.

Therefore there exists a coset $t+H$ of $H=\langle d_1 \rangle$ such that
$$\frac{|H|}{2} \geq |A \cap (t+H)| \geq |H|-3.$$

This means that $q/d_1 = |H| \leq 6$, and so $l d_1 \equiv 0$ mod $q$ for some $1 \leq l \leq 6$, and any arithmetic progression of difference $d_1$ forming $A$ could not have more that $5$ elements, implying that $m \leq 15$.

To prove Corollary \ref{digsetcorollary}, thanks to Theorem \ref{digsetteo} we are left to cosider the case of $A=P_1 \cup P_2$ a proper union of two arithmetic progressions of common difference $d$, and $2A \subseteq \{x, y\} +A$.

Once we establish that such a set cannot be a digital set, we are done since the only possibilities for a single arithmetic progression to be a digital set with minimal number of distinct carries are clearly the ones stated in the corollary.

Consider at first the case $(d, q) > 1$. After a dilation, thanks to Lemma \ref{redd2}, we can assume $d|q$. Since $A$ is a digital set, we must have $d=2$ and hence $2|q$.

Moreover, by  \ref{sg1} we have $|P_1|=|P_2|= m/2$, and $P_i = \alpha_i + 2\cdot[0, m/2-1]$, $i=1,2$ , where $\alpha_1  \not\equiv \alpha_2$ mod $2$.

Then
$$2A=  (2\alpha_1 + 2\cdot[0, m-1]) \cup (\alpha_1 + \alpha_2 +  2\cdot[0,m-1]) \cup (2\alpha_2 +  2\cdot[0, m-1]).$$

By the parity of $\alpha_1$ and $\alpha_2$, we must have $$|(2\alpha_1 + 2\cdot[0,m-1]) \cup  (2\alpha_2 + 2\cdot[0,m-1])| \leq m+1,$$

which implies without loss of generality, since $2m \leq q$, that $2\alpha_1 \in \{2\alpha_2, 2\alpha_2 + 2, 2\alpha_2 + 4 \}$.

Once again, since $\alpha_1 \not\equiv \alpha_2$ mod $2$, and $A$ is not an arithmetic progression, this leaves us with the only choice $\alpha_1 = \alpha_2 + 1 + q/2$, and so, up to translation,
 $$A=2\cdot\left[0, \frac{m}{2}-1\right] \cup \left(\frac{q}{2}+1+2\cdot \left[0, \frac{m}{2}-1\right]\right),$$

which is a single arithmetic progression of difference $q/2+1$.

Assume now $(d, q)=1$, so that, after a dilation and a translation, we can assume that $A$ is of the form
$$A = [0, a-1] \cup [bm+a, (b+1)m -1],$$
with $a \geq m-a$, $1 \leq b \leq q/m -2$.

 Then $2A = B_1 \cup B_2 \cup B_3$, where
$$
B_1 = [0, 2a-2], \quad B_2 = [bm+a, (b+1)m+a-2], \quad B_3 = [2bm + 2a, 2(b+1)m -2]
$$
are all non empty sets.

A routine check shows that $B_1 \cap B_2 = \emptyset$, and $|B_1|+|B_2|= 2a+m-2 \leq 2m$ implies $m/2 \leq a \leq (m+2)/2$ and $|B_3 \cap (B_1 \cup B_2)^c| \leq 2$.

For these possible values of $a$, we must have $B_3 \subseteq B_1$, so that $m(2b+1) \equiv 0$ mod $q$, and since all primes dividing $m$ must divide $q/m$, we have $2 \nmid q$ and so
$$a=\frac{m+1}{2},bm= \frac{q-m}{2} \Longrightarrow A=\left[ 0, \frac{m-1}{2}\right] \cup \left[ \frac{q+1}{2}, \frac{q+m-2}{2}\right].$$
Once again, this is a single arithmetic progression of difference $(q+1)/2$.
\end{section}

\section*{}
\bibliography{bibliography}
     \end{document}